\documentclass[12pt,leqno,amsfonts,amscd]{amsart}
\usepackage{epsfig}
\usepackage{amsmath}
\usepackage{amsfonts}
\usepackage{amssymb}
\usepackage{color}
\usepackage{hyperref}

\newtheorem{theorem}{Theorem}[section]
\newtheorem{lemma}[theorem]{Lemma}
\newtheorem{prop}[theorem]{Proposition}
\newtheorem{cor}[theorem]{Corollary}

\theoremstyle{definition}
\newtheorem{definition}[theorem]{Definition}

\theoremstyle{remark}

\newtheorem{remark}[theorem]{Remark}
\newtheorem{remarks}[theorem]{Remarks}
\numberwithin{equation}{section}

\newcommand{\calb}{{\mathcal B}}
\newcommand{\cald}{{\mathcal D}}
\newcommand{\calm}{{\mathcal M}}
\newcommand{\calf}{{\mathcal F}}
\newcommand{\calh}{{\mathcal H}}

\newcommand{\RR}{{\mathbb R}}
\newcommand{\CC}{{\mathbb C}}

 \DeclareMathOperator{\psh}{PSH}
 
 \DeclareMathOperator{\FPSH}{\calf-PSH}
 \DeclareMathOperator{\FCPSH}{\calf-CPSH}
 
\DeclareMathOperator{\QB}{QB}

\let\PSH=\psh

\let\cal=\mathcal
\renewcommand{\phi}{\varphi}

\newcommand{\ip}[2]{\langle #1,#2\rangle}

\begin{document}
\title[$\calf$-Plurisubharmonic functions and the M-A operator]{Plurifinely Plurisubharmonic functions and the Monge Amp\`ere Operator}

\author{Mohamed El Kadiri}
\address{
Universit\'e Mohammed V-Agdal \\
D\'epartement de Math\'ematiques
\\Facult\'e des Sciences\\ B.P. 1014, Rabat\\ Morocco
}
\email{elkadiri@fsr.ac.ma}

\author{Jan Wiegerinck}
\address{KdV Institute for Mathematics
\\University of Amsterdam
\\Science Park 904
\\P.O. box 94248, 1090 GE Amsterdam
\\The Netherlands}
\email{j.j.o.o.wiegerinck@uva.nl}

\begin{abstract}We will define the Monge-Amp\`ere operator on finite (weakly) plurifinely plurisubharmonic functions in plurifinely open sets $U\subset\CC^n$ and show that it defines a positive measure.
Ingredients of the proof include a direct proof for bounded strongly plurifinely plurisubharmonic functions, which is based on the fact that such functions can plurifinely locally be written as difference of ordinary plurisubharmonic functions, and an approximation result stating that in the Dirichlet norm weakly plurifinely plurisubharmonic functions are locally limits of plurisubharmonic functions.
As a consequence of the latter, weakly plurifinely plurisubharmonic functions are strongly plurifinely plurisubharmonic outside of a pluripolar set.
\end{abstract}
\maketitle

\section{Introduction}

The plurifine topology $\calf$ on a Euclidean open set $\Omega\subset\CC^n$ is the smallest topology that makes all plurisubharmonic function on $\Omega$ continuous. The construction is completely analogous to the better known fine topology in classical potential theory of H. Cartan. Good references for the latter are \cite{AG, D}. The topology $\calf$ was introduced in \cite{F6}, and studied e.g. by Bedford and Taylor in \cite{BT}, and El Marzguioui and the second author in \cite{E-W1, EW1}.  Notions related to the topology $\calf$ are provided with the prefix $\calf$, e.g. an $\calf$-domain is an $\calf$-open set that is connected in $\calf$.

Just as one introduces finely subharmonic functions on fine domains in $\RR^n$, cf.~Fuglede's book \cite{F1}, one can introduce plurifinely plurisubharmonic functions on $\calf$-domains in $\CC^n$. In case $n=1$ these are just finely subharmonic functions on fine domains in $\RR^2$.
From now on we will assume $n>1$. Then there are two variants, which are up to now only known to be equal in case $n=1$, cf.~\cite{F1}.  \emph{Strongly plurifinely plurisubharmonic functions} are defined as $\calf$-locally decreasing limits of $C$-strongly plurifinely plurisubharmonic functions, which in turn are defined as $\calf$-locally uniform limits of continuous plurisubharmonic functions defined in shrinking Euclidean neighborhoods of the $\calf$-neighborhood under consideration, cf.~ Definition \ref{def3.3}. \emph{Weakly plurifinely plurisubharmonic functions} on the other hand, are defined as $\calf$-upper semicontinuous functions on an $\calf$-open set, with the property that their restrictions to intersections with complex lines are finely subharmonic.  These functions were defined and studied by the first author in \cite{EK} and by El Marzguioui and the second author in \cite{EW1, EW2}. In a joint recent paper with Fuglede, \cite{EFW}, it is observed that ``strong'' implies ``weak'', and many common properties of plurisubharmonic functions are proven for weakly plurifinely plurisubharmonic functions. An overview of all this is in \cite{Wie2012}.

In the sequel we will be mostly concerned with weakly plurifinely plurisubharmonic functions, for which we use the notation $\calf$-plurisubharmonic functions. The cone of $\calf$-plurisubharmonic functions on an $\calf$-open set $U$ is denoted by $\FPSH(U)$.

In Section \ref{sec2} we will establish some approximation results to the effect that weakly plurifinely plurisubharmonic functions can be approximated $\calf$-locally in the Dirichlet norm  by  plurisubharmonic functions. Moreover, if $f\in \FPSH(U)$, then there exists a pluripolar set $E$ such that on $U\setminus E$, each point admits an $\calf$-neighborhood on which $f$ is strongly $\calf$-plurisubharmonic. Note that at present we are still unable to prove that strongly and weakly finely plurisubharmonic functions are the same.

In Section \ref{sec1} we will give a definition of the $\calf$-local Monge-Amp\`ere mass of a finite $f\in\FPSH(U)$, where $U$ is an $\calf$-domain in $\CC^n$, $n\ge 2$. The idea is to use the fact that $f$ can $\calf$-locally at $z\in U$ be written as $u-v$ where $u,v$ are bounded plurisubharmonic functions defined on a ball about $z$, cf.~\cite{EW2, EFW}. For such differences of plurisubharmonic functions the Monge-Amp\`ere may be defined by multilinearity, cf.\ Cegrell and Wiklund,~\cite{CW},

\[(dd^c(u-v))^n=\sum_{p=0}^n\binom{n}{p}(-1)^p(dd^cu)^{n-p}\wedge(dd^cv)^p.
\]
We will show that this definition is independent of the choice of $u$ and $v$.
Thus an $\calf$-local definition of $(dd^cf)^n$ is obtained.

 In Section \ref{sec3} we show that the Monge-Amp\`ere of a finite $\calf$-plurisubharmonic function can be defined and is a positive measure. This is done at first $\calf$-locally for $C$-strongly $\calf$-plurisubharmonic functions.  The results in Section \ref{sec1} combined with the facts that $\calf$ has the quasi-Lindel\"of property and that the Monge-Amp\`ere of bounded plurisubharmonic functions does not charge pluripolar sets, lead to a globally defined positive Monge-Amp\`ere mass. The results of Section \ref{sec2} then allow to pass to finite $\calf$-plurisubharmonic functions. For finite plurisubharmonic functions $u$ on Euclidean domains, we recover the \emph{nonpolar part} $NP(dd^cu)^n$ of the Monge-Amp\`ere measure as defined in \cite[P.236]{BT}. Let us recall that in general this Monge-Amp\`ere mass need not be a Radon measure.

For a set  $A\subset\CC^n$
we  write $\overline A$ for the closure of $A$ in the one point compactification of ${\CC}^n$, $\tilde{A}$ for the $\calf$-closure of
$A$, $\partial_{\cal F}A$ for the $\calf$-boundary
of $A$, and $\partial_fA$ for the fine boundary of $A$
in ${\bf C}^n \cong {\bf R}^{2n}$.

We mention some recent results that we will draw on in this paper.
From \cite{EW2} we will often use that $\calf$-locally bounded ${\cal F}$-plurisubharmonic functions are ${\cal F}$-locally the difference of bounded plurisubharmonic functions defined on a Euclidean open set and its consequence that $\calf$-plurisubharmonic functions are $\calf$-continuous.
In \cite{EFW} many of the classical properties of plurisubharmonic functions are extended to the plurifine situation. We will use in particular that bounded finely subharmonic functions that remain finely subharmonic under composition with complex affine mappings, are $\calf$-plurisubharmonic. We mention also that $\calf$-plurisubharmonic functions are invariant under holomorphic transformations, and that the upper envelop of a locally bounded family of $\calf$-plurisubharmonic functions differs from its $\calf$-upper semicontinuous regularisation at most on a pluripolar set.

\subsection*{Acknowledgement} We are grateful to Iris Smit, who spotted a mistake in an early version of this paper, and to Bent Fuglede and the referee for comments that led to improvements in Section \ref{sec2}.

\section{Local approximation of $\calf$-plurisubharmonic functions}
\label{sec2}
We will prove several density results \ref{thm2.5}, \ref{thm2.5a} for $\calf$-plurisubharmonic functions. Inspired by \cite{F2}, we will employ duality with respect to a pairing $\calb$ of Hilbert spaces. Contrary to the situation in \cite{F2}, the pairing will not be \emph{separated} --we will use throughout the terminology of \cite{KN} --  but this causes no problems.

We denote by $\calm_{n,m}$ the space of complex $n\times m$ matrices and by $\calh_n$ the space of Hermitian matrices in $\calm_{n,n}$. Next $\calh_n^+$ denotes the cone in $\calh_n$ consisting of positive Hermitian matrices.

In this section $\Omega\subset\CC^n=\RR^{2n}$ will be a bounded, Euclidean open domain and $U$ will be an $\calf$-open subset of $\Omega$. We denote by $C^\infty_c(\Omega)$ the space of $C^\infty$ functions on $\Omega$ with compact support in $\Omega$, and by $\cald_1(\Omega)$ the closure of $C^\infty_c(\Omega)$ in the Dirichlet norm $\|u\|=\|\nabla u\|_2$. It is a Hilbert space with inner product $\ip{u}{v}=\int_\Omega \nabla u\nabla v \, dV$. An element of $\cald_1(\Omega)$ is an equivalence class of functions, which contains a quasi-continuous function, cf.~\cite{DL}. The subspace of $\cald_1(\Omega)$ consisting of (equivalence classes of) functions that are quasi-continuous and 0 a.e. in $\Omega\setminus U$ is denoted by $\cald_1^0(\Omega,U)$. Here functions are equivalent if they are equal a.e., hence elements of $\cald_1^0(\Omega, U)$ may be represented by functions which are everywhere 0 on $\Omega\setminus U$. We denote the cones of non-negative elements in $\cald_1(\Omega)$ and $\cald_1^0(\Omega,U)$ by $\cald_1^+(\Omega)$ and $\cald^{+0}_1(\Omega,U)$, respectively. 

Next we introduce cones of (plurifinely) plurisubharmonic functions. We let $\FPSH(\Omega,U)$ be the cone
of functions in ${\cal D}_1(\Omega)$ whose restriction
to $U$ is in $\FPSH(U)$, $\FPSH^-(\Omega,U)=\FPSH(\Omega, U)\cap -\cald_1^+(\Omega)$, and $\FPSH^{-0}(\Omega,U)=\FPSH(\Omega, U)\cap -\cald_1^{+0}(\Omega,U)$.
For $\omega$ a Euclidean open in $\Omega$, we introduce like spaces $\PSH(\omega,\Omega)$, $\PSH^-(\omega,\Omega)$, and $\PSH^{-0}(\omega,\Omega)$
of functions that are, respectively, in $\PSH(\omega)\cap\cald_1(\Omega)$, in $\PSH(\omega)\cap-\cald_1^{+}(\Omega)$, and  in $\PSH(\omega)\cap-\cald_1^{+0}(\omega,\Omega)$.

\smallskip
From \cite{EFW} we recall the following theorem.
\begin{theorem}
 \label{thm2.1}\cite[Theorem 4.2]{EFW}
Suppose that $u\in\cald_1(\Omega)$ with values in $[-\infty,\infty[$ is $\calf$-continuous on $U$. Then the following are equivalent.
\begin {itemize}
 \item[a.] $u$ is $\calf$-plurisubharmonic in $U$.
 \item[b.]  For every $\Lambda=(\lambda_1,...,\lambda_n)\in {\CC}^n$ and every
$\varphi\in {\cald}_1^+(\Omega,U)$
\[
-\sum_{j,k=1}^n\lambda_j{\overline {\lambda_k}}\int_U\partial_ju{\overline \partial}_k\varphi\, dV\ge 0,
\]
where $dV$ is  Lebesgue measure on ${\CC}^n\cong {\RR}^{2n}$.
\end {itemize}
\end{theorem}

We will use this result in a different formulation.

\begin{theorem}\label{thm2.2}
Let
$u\in {\cal D}_1(\Omega)$ with values in
$[-\infty,+\infty[$
and ${\cal F}$-continuous on $U$. Then the following conditions are equivalent:
\begin{itemize}
\item[a.] $u$ is $\FPSH(U)$.

\item[b.] For every $M=(a_{ij})\in {\calh}_n^+$  and all
$\varphi\in {\cal D}_1^+(\Omega, U)$
\[
-\sum_{j,k=1}^na_{ij}\int_U\partial_iu{\overline \partial}_j \varphi dV\ge 0.
\]
\end{itemize}
\end{theorem}

\begin{proof}
 b)$\Longrightarrow$ a): Let $\Lambda=(\lambda_1,...,\lambda_n)\in {\CC}^n$.
We view  $\Lambda$ as an element of $\calm_{n,1}$. Let $\Lambda^*$ be the adjoint of $\Lambda$. Then
$\Lambda\Lambda^*=(\lambda_i{\overline \lambda}_j)\in\calm_{n,n}$.
Now b) implies that for all $\varphi\in {\cal D}_1^+(\Omega,U)$ one has
\[
-\sum_{i,j=1}^n\lambda_i{\overline {\lambda_j}}\int_U\partial_iu{\overline \partial}_j\varphi dV\ge 0,
\]
hence condition a) is satisfied.

a)$\Longrightarrow $b): Let
$u\in {\cal D}_1(\Omega)$ belong to $\FPSH(U)$, $\varphi \in {\cal D}_1^+(\Omega,U)$ and
$M=(a_{ij})\in {\calh_n}^+$. Let  $\alpha_1,...,\alpha_n\ge 0$ be the eigenvalues of $M$, counted with multiplicity, and let
$ \{\Lambda^1,...\Lambda^n\}$ be an orthonormal basis of of eigenvectors of $M$ in
${\CC}^n$ such that
\[M=\sum_{k=1}^n\alpha_k\Lambda^k{\Lambda^k}^*.
\]
If a) holds, then
\[
-\sum_{i,j}a_{ij}\int_U\partial_iu{\overline \partial}_j\varphi dV
=-\sum_k \alpha_k\sum_{ij}\lambda^k_i{\overline {\lambda^k_j}}\int_U\partial_iu{\overline \partial}_j\varphi dV\ge 0
\]
according to Theorem \ref{thm2.1}, hence assertion b).
\end{proof}

Let $X$ be the Hilbert space ${\calh}_n\otimes {\cal D}_1(\Omega)$ endowed with the inner product inherited from ${\cal D}_1(\Omega)$ and let $\calb$ be the ${\RR}$-bilinear form on
$X \times {\cal D}_1(\Omega)$
defined by
\[
\calb(M\otimes u,v)=
-\sum_{j,k=1}^n a_{jk}\int_\Omega\partial_ju{\overline \partial}_k v \, dV
\]
for all $ M\otimes u\in X$
and all $v\in {\cal D}_1(\Omega)$, where $M=(a_{jk})$. We have
\begin{lemma}\label{lem2.5a}
The form $\calb$ is continuous on  $X\times
{\cal D}_1(\Omega)$.

 The map $x\mapsto \calb(x,\cdot)$ maps $X$ continuously \emph{onto} the dual space ${\cal D}_1(\Omega)^*$.
 \end{lemma}
 \begin{proof} The first statement follows from
\[\left|\calb(M\otimes u,v)\right|\le\sum_{j,k=1}^n|a_{jk}|\int|\nabla u||\nabla v|\, dV\le n
\|M\|_2\|u\|\|v\|.\]
 For the second statement,  just note that for every $u\in \cald_1(\Omega)$ we have $\ip{u}{v}=\calb(I\otimes u,v)$.
 \end{proof}

 \begin{remark}
 It follows immediately that if $\calb(u,v)=0$ for every $u\in X$ then $v=0$, i.e. $X$ \emph{distinguishes points} in $\cald_1(\Omega)$. But $\cald_1(\Omega)$ does not distinguish points in $X$, hence $\calb$ is not separated: Indeed, take any $g\in C^\infty_0(\Omega)$ not identically 0, $u_1=\frac{\partial^2 g}{\partial z_1\partial\overline{z_1}}$, $u_2=\frac{\partial^2 g}{\partial z_2\partial{\overline{z_2}}}$, and let $m^1_{22}=1$, $m^1_{ij}=0$ if $(i,j)\ne(2,2)$, and $m^2_{11}=1$, $m^1_{ij}=0$ if $(i,j)\ne(1,1)$. Then for every $v\in\cald_1(\Omega)$
\[\calb(m^1\otimes u_1-m^2\otimes u_2,v)=0.\]
\end{remark}

Let $B$ be a pairing of topological vector spaces $E$ and $F$. We denote by $w(E,F)$ the weak topology induced  by $F$ via $B$ on $E$ and similarly by $w(F,E)$ the weak topology induced by $E$ via $B$ on $F$.
The dual of a cone $A\subset E$ is denoted by
\[A^*=\{f\in F: \forall x\in A,\ B(x,f)\le 0\}.\]
 Similarly, the dual of a cone $Z$ in $F$ is the subset of $x\in E$ such that for all $f\in Z$, $B(x,f)\le 0$, and is denoted by $Z_*$.
 We will write $A^*_{\ *}$ for $(A^*)_*$.
We now adapt Theorem 16.3 from \cite{KN} to our needs, and we provide its very similar proof for convenience of the reader.

\begin{theorem}\label{thm2.2a}
Let $B$ be a pairing of topological vector spaces $E$ and $F$.
\begin{enumerate}
\item For each cone $A\subset E$, $A^*$ is an $w(E,F)$-closed, convex cone.
\item If $A\subset B$ are cones  in $E$ then $B^*\subset A^*$.
\item If $A\subset E$ is non empty, then $A^*_{\ *}$ is the smallest convex, $w(E,F)$-closed cone in $E$ that contains $A$.
\item For  a nonempty cone $A\subset E$ we have ${A^*_{\ *}}^*=A^*$.
\item If $\{A_t\}$ is a family of cones in $E$, then $(\cup_tA_t)^*=\cap_t A_t^*$.
\item If $\{A_t\}$ is a family of convex $w(E,F)$-closed cones in $E$, then $(\cap_tA_t)^*$ is the smallest convex, $w(F,E)$-closed cone in $F$ that contains $A_t^*$ for all $t$.
\end{enumerate}
\end{theorem}

\begin{proof} Statements (1) and (2) are obvious. For (3), let $C$ be the smallest $w(E,F)$-closed cone that contains $A$, then $C\subset A^*_{\ *}$. If $x\notin C$ then in view of the Hahn-Banach theorem,  there exists a $w(E,F)$ continuous functional $g$ on $E$ such that $g| C\le 0$ and $g(x)=1$. There exists $f\in F$ with $g(y)=B(y,f)$ for $y\in E$, cf.~\cite[Theorem 16.2]{KN}.
Then $f\in A^*$, hence $x\notin A^*_{\ *}$. Next, (4) follows immediately from (1) and the dual of (3). Now (5) is clear, as $B(x,f)\le 0$ for $x\in \cup_tA_t$ if and only if $f\in\cap_tA_t^*$. Finally, by (3) $A_t={A_t}^*_{\ *}$, so that by (5)
\[(\cap_tA_t)^*=(\cap_t{A_t}^*_{\ *})^*=((\cup_t {A_t}^*)_*)^*,\]
which equals the $w(E,F)$-closure of the convex hull of $\cup_tA_t^*$, and (6) is proven.
\end{proof}

We write $\Gamma(U)$ for the convex cone
${\calh}_n^+ \otimes{\cal D}_1^{+0}(\Omega,U)\subset
X$ generated by the
elements of $X$ of the form $M\otimes u$
where $M\in {\calh}_n^+$ and $u\in {\cal D}_1^{+0}(\Omega,U)$.

 Next we introduce
${\cal C}_{\cal F}(U)$ the set of
${\cal F}$-continuous functions on $U$ with values in
$[-\infty,+\infty[$. Let  $\Gamma(U)^*$ be the dual
of the convex cone $\Gamma(U)\subset X$ relative to $\calb$.
Then according to Theorem \ref{thm2.2}, one has
\begin{equation}\label{eq2.2}\Gamma(U)^*\cap{\cal C}_{\cal F}(U)=\FPSH(\Omega,U).
\end{equation}

As this holds for all $\calf$-open $U$ in $\Omega$, we have in view of \cite[Proposition 3.14]{EFW} for a Euclidean open $\omega\subset \Omega$, that $\FPSH(\Omega,\omega)=\PSH(\Omega,\omega)$.

\begin{remarks}\label{rem2.3}\ In view of \cite[Th\'eor\`eme 11]{F2}, respectively \cite[Proposition 6]{F2} and Theorem \ref{thm2.2} the following hold:
\begin{itemize}
\item[a.] We have $\Gamma(U)^*\subset {\cal S}(\Omega,U)$,
the cone of functions in ${\cal D}_1(\Omega)$ that are
${\RR}^{2n}$-q.e. subharmonic in $U$ .
\item[b.] For $v\in {\cal D}_1(\Omega)$ The statements \lq$v$  belongs to
 $\Gamma(\omega)^*$\rq\ and \lq for all
$\Lambda=(\lambda_1,...,\lambda_n)\in {\bf C}^n$, the  distribution
$\sum_{j,k=1}^n\lambda_j{\overline {\lambda_k}}
\frac{\partial^2v}{{\partial z_j}{\partial {\overline z_k}}}$ is a positive measure
on $\omega$\rq\ are equivalent.
\end{itemize}
\end{remarks}
Recall from
\cite[Proposition 8]{F2} that for finely open sets $U\subset\Omega$, hence \`a fortiori for plurifinely open set $U\subset\Omega$:
\begin{equation}\label{eq2.3a}
{\cal D}_1^{+0}(\Omega,U)=\cap_\omega{\cal D}_1^{+0}(\Omega,\omega).
\end{equation}
Here the intersection runs over all (Euclidean) open sets $\omega$ with
$U\subset \omega\subset \Omega$.

\begin{prop}\label{prop2.3} For all open sets
$\omega\subset \Omega$, we have
\[\Gamma(\omega)^*=\PSH(\Omega,\omega).
\]
\end{prop}
\begin{proof} The inclusion $\PSH(\Omega,\omega)\subset \Gamma(\omega)^*$
is contained in \eqref{eq2.2}.
For the other direction, let  $f\in \Gamma(\omega)^*$.
According to Remarks \ref{rem2.3}-a we have $\Gamma(\omega)^*\subset {\cal S}(\Omega,\omega)$,
hence $f$  is, or rather admits a representative, which is upper semicontinuous in $\omega$,
therefore, in view of Remarks \ref{rem2.3}-b, $f\in \PSH(\omega)$. We conclude that $f\in \PSH(\Omega,\omega)$.
\end{proof}

\begin{prop}\label{prop2.4} For all $\calf$-open
$U\subset \Omega$, we have
$$\Gamma(U)=\cap_\omega \Gamma(\omega),$$
where the intersection is taken over all Euclidean open sets $\omega$ with
$U\subset \omega\subset \Omega$.
\end{prop}

\begin{proof} The inclusion $\Gamma(U)\subset \cap_\omega \Gamma(\omega)$
is clear. For the other inclusion we proceed as follows.
Let $x=\sum_{j\in J} M_j\otimes u_j\in
\cap_\omega \Gamma(\omega)\subset X$, with $J$ finite, $M_j\in {\calh}_n^+\setminus\{0\}$,
and $ u_j\in {\cal D}_1^+(\Omega)$ for all
$j\in J$. Let $g$ be the  ${\RR}$-linear form on ${\calh}_n$
defined by  $g( M)=tr(M)=\sum_{i=1}^na_{ii},$
for all $M=(a_{ij})\in {\calh}_n$. Then $g(M)> 0$ where
$M\in {\cal H}_n^+\setminus \{0\}$.

Now let $\omega\subset\Omega$ be open and such that $U \subset \omega$. Because $x\in {\Gamma}(\omega)$, we can write
$x=\sum_{k\in K} L_k \otimes \tilde u_k\in
\Gamma(\omega)$ with $K$ finite, $L_k\in {\calh}_n^+$, and
$\tilde u_k \in {\cal D}_1^{+0}(\Omega,\omega)$
for all $k\in K$. For every linear
continuous form $h$ on ${\cal D}_1(\Omega)$, we have
\[h(\sum_{j\in J}tr(M_j) u_j)=(g\otimes h)x
=h(\sum_{k\in K}tr(L_k) \tilde u_k).\]
Therefore
\[\sum_{j\in J}tr(M_j) u_j
=\sum_{k\in K}tr(L_k) \tilde u_k.\]
As the functions
$\tilde u_k$ are q-e. $0$ in $\Omega\setminus \omega$ and because $tr(M_j), tr(L_k)>0$ and the functions $u_j$
are non-negative, we infer that the latter are also 0 q-e.\ in  $\Omega\setminus \omega$.
Hence for all
$j\in J$, we have $u_j\in \cap_{\omega} {\cal D}_1^{+0}(\Omega,\omega)$, which equals $\cald^{+0}(\Omega, U)$ by \eqref{eq2.3a}. We conclude that
$x\in {\Gamma}(U)$. Hence
$\Gamma(U)=\cap_{\omega}\Gamma(\omega)$.
\end{proof}

\begin{theorem}\label{thm2.5}
\[\FPSH(\Omega,U)
={\overline {\cup_\omega \PSH(\Omega,\omega)}}\cap {\cal C}_{\cal F}(U),
\]
where the union is over all Euclidean open $\omega$ with
$U\subset \omega\subset \Omega$, and the closure is in the
 sense of the strong (i.e.\ norm) topology in the Hilbert space
  ${\cal D}_1(\Omega)$.
\end{theorem}

\begin{proof} Proposition \ref{prop2.4} gives $\Gamma(U)=\cap_{\omega}\Gamma(\omega)$. Theorem \ref{thm2.2a} (6)
yields that
\[\FPSH(\Omega,U)=\Gamma(U)^*\cap {\cal C}_{\cal F}(U)=(\cap_\omega \Gamma(\omega))^*\cap{\cal C}_{\cal F}(U)\]
\[={\overline {\cup_\omega \Gamma(\omega)^*}}\cap{\cal C}_{\cal F}(U)
={\overline {\cup_\omega \PSH(\Omega,\omega)}}\cap{\cal C}_{\cal F}(U).
\]
Here at first the closure of the convex set $\cup_\omega \PSH(\Omega,\omega)$ is in $w(\cald_1(\Omega), X)$, the topology induced by $X$ on ${\cal D}_1(\Omega)$. This topology coincides with the usual weak topology on $\cald_1(\Omega)$ because of Lemma \ref{lem2.5a} and next, by the Hahn-Banach theorem this closure equals the closure in the strong topology in ${\cal D}_1(\Omega)$. This proves the theorem.
\end{proof}

We next present more useful ``bounded'' versions of Theorem \ref{thm2.5}.

\begin{theorem}\label{thm2.5a}
\begin{equation}\label{eq2.5a2}\FPSH^-(\Omega,U)
=\overline {\cup_\omega \PSH^{-}(\Omega,\omega)}\cap {\cal C}_{\cal F}(U),
\end{equation}
\begin{equation}\label{eq2.5a1}\FPSH^{-0}(\Omega,U)
=\overline {\cup_\omega \PSH^{-0}(\Omega,\omega)}\cap {\cal C}_{\cal F}(U),
\end{equation}

where again unions are taken over all Euclidean open $\omega$ with
$U\subset \omega\subset \Omega$, and the closure is in the
 sense of the strong (i.e.\ norm) topology in the Hilbert space
  ${\cal D}_1(\Omega)$.
\end{theorem}
\begin{proof} For both cases it suffices to prove the inclusion ``$\subset$''.  The proof of \eqref{eq2.5a2}
is a straightforward adaptation of the proof of Theorem \ref{thm2.5}.

\[\begin{split}\FPSH^-(\Omega,U)=\FPSH(\Omega,U)\cap-\cald^+(\Omega)=\left(\Gamma(U)^*\cap-\cald^+(\Omega)\right)\cap{\cal C}_\calf(U)\\=
\left(\Gamma(U)\cup-\cald^+(\Omega)_*\right)^*\cap{\cal C}_\calf(U)=
\left(\cap_\omega (\Gamma(\omega)\cup-\cald^+(\Omega)_*)\right)^*\cap{\cal C}_\calf(U)\\=
{\overline {\cup_\omega \left(\PSH(\Omega,\omega)\cap-\cald^+(\Omega)\right)}}\cap{\cal C}_{\cal F}(U)=
{\overline {\cup_\omega \PSH^-(\Omega,\omega)}}\cap{\cal C}_{\cal F}(U).
\end{split}\]

For the proof of \eqref{eq2.5a1} we proceed as follows. For every Euclidean open $\omega$ with $U\subset\omega\subset\Omega$ we have that $\cald_1^{+0}(\Omega, U)\subset \cald^{+0}_1(\Omega,\omega)$, so that $\cald_1^{+0}(\Omega, \omega)_*\subset \cald^{+0}_1(\Omega,U)_*$. Hence
\[\Gamma(\omega)\cup-\cald_1^{+0}(\Omega, \omega)_*\subset \Gamma(\omega)\cup -\cald^{+0}_1(\Omega,U)_*.
\]
We infer that
\[\cup_\omega (\Gamma(\omega)\cup\-\cald_1^{+0}(\Omega, \omega)_*)\subset\cap_\omega \Gamma(\omega)\cup -\cald^{+0}_1(\Omega,U)_*,
\]
where the union is taken over all Euclidean open $\omega$ with
$U\subset \omega\subset \Omega$. From this we have
\[
\left(\cap_\omega \Gamma(\omega)\cup -\cald^{+0}_1(\Omega,U)_*\right)^*\subset\left(\cup_\omega \Gamma(\omega)\cup-\cald_1^{+0}(\Omega, \omega)_*\right)^*,
\]
hence,
\[
\left(\Gamma(U)\cup -\cald^{+0}_1(\Omega,U)_*\right)^*\subset\left(\cup_\omega \Gamma(\omega)\cup-\cald_1^{+0}(\Omega, \omega)_*\right)^*.
\]
It follows that
\[\FPSH^{-0}(\Omega,U)=\Gamma(U)^*\cap-D_1^{+0}(\Omega,U)
\subset \overline {\cup_\omega \PSH^{-0}(\Omega,\omega)}\cap {\cal C}_{\cal F}(U).
\]
As in the proof of Theorem \ref{thm2.5}, in both cases the closure is at first in the weak and then, by Hahn-Banach, also in the strong topology.
\end{proof}

\smallskip
In the next two corollaries we assume that $U$ is compactly contained in $\Omega$.

\begin{cor}\label{cor2.1} Let $u\in \FPSH(\Omega,U)$ be bounded from above by $M$ on $U$. Then there exist
a sequence $(\omega_j)$ of open sets  $U\subset\omega_j\subset\Omega$
and functions
$f_j\in \FPSH(\Omega,\omega_j)$, $f_j\le M$ on $\omega_j$ such that $(f_j)$ converges to $u$
in the strong topology of ${\cal D}_1(\Omega)$.
\end{cor}
\begin{proof}
Let $\chi \in C^\infty_0(\Omega)$ be $\le 0$ and $=-M$ on a Euclidean neighborhood $\omega$ of $U$.
Then $u+\chi\in \FPSH^-(\Omega,U)$. Theorem \ref{thm2.5a} \eqref{eq2.5a2} gives a sequence of functions $u_j\in\FPSH^-(\Omega,\tilde \omega_j)$ converging strongly to $u+\chi$. Set $\omega_j=\omega\cap\tilde\omega_j$. Then the functions $f_j=u_j-\chi$ are in $ \FPSH^{-}(\Omega,\omega_j)$, $f_j\le M$ on $\omega_j$, and $(f_j)$ converges to $u$ strongly.
\end{proof}

\begin{cor}\label{cor2.2} Let $u\in \FPSH(\Omega,U)$ be bounded from above by $M$ on $U$. Then there exist
a sequence $(\omega_j)$ of open sets  $U\subset\omega_j\subset\Omega$
and functions
$f_j\in \PSH(\omega_j)$, $f_j\le M$ such that $(f_j)$ converges to $u$ q-e. on $U$.
\end{cor}

\begin{proof} Modifying $u$ as in the previous proof, the corollary follows immediately from
 Theorem \ref{thm2.5a} and \cite[Th\'eor\`eme 4.1, p.357]{DL}.
\end{proof}

\begin{lemma}\label{lem2.6} Let $(u_i)$ be a  family of ${\cal F}$-locally
uniformly bounded  ${\cal F}$-pluri\-subharmonic functions
on $U$. Then for every
$z\in U$ there exist an open ball $B=B(z,r)$ and
an ${\cal F}$-open ${\cal O}$ such that $z\in {\cal O}\subset B$ and
$u_i|_B\in \FPSH(B,{\cal O})$.
\end{lemma}

\begin{proof} Following the proof of \cite[Theorem 2.4]{EFW}, there exist an open ball  $B=B(z,r)$,
an ${\cal F}$-open ${\cal O}$, a family $(v_i)$ of
uniformly bounded plurisubharmonic functions
on $B$, and a $\Phi\in\psh(B)$ such that for all $i\in I$,
$u_i=v_i-\Phi$ on ${\cal O}$. By adding the same constant to each of the functions $-v_i$ and $-\Phi$ one may suppose that these are  positive on $B$. After replacing $B$ by
 $B'=B(z,\frac{r}{2})$ and $-v_i$ as well as
$-\Phi$ by their  swept out on $B'$, and shrinking $\cal O$ if necessary, one can assume that $-v_i$ and $-\Phi$ are potentials on $B$. Then one has
$u_i=v_i-\Phi\in {\cal D}_1(B)$.
\end{proof}

\begin{theorem}\label{thm2.7} Let $u\in\FPSH(U)$ be finite.
Then for all $z\in U$ there exist
an ${\cal F}$-open neighborhood ${\cal O}$, a constant $M$, a sequence of open sets $(\omega_j)$ such that ${\cal O}\subset \omega_j\subset\Omega$, and a sequence of functions $(f_j)$, $f_j\in \PSH(\omega_j)$ such that $f\le M$ on $\cal O$, $f_j\le M$ on  $\omega_j$, and $(f_j)$ converges q.e. to $u$.
\end{theorem}

\begin{proof} Let $z\in U$. Because $u$ is $\calf$-continuous, there exists $M>0$ such that $u\le M$ on an $\calf$-neighborhood of $z$. The preceding lemma applied to $u$ provides us then with
an ${\cal F}$-open neighborhood ${\cal O}$ such that
$u\in \PSH(\Omega,{\cal O})$ and $u\le M$ on $\cal O$. The result now follows immediately from
 Corollary \ref{cor2.2}.
\end{proof}

\begin{cor}\label{cor2.7} Let $u$ be a finite function in $\FPSH(U)$. Then each point $z$ of $U$ admits a plurifine neighborhood ${\cal O}$ and
a sequence $(f_j)$ of plurisubharmonic  functions defined and uniformly bounded on (shrinking) neighborhoods of ${\cal O}$
such that $u=\inf_k(\sup_{j\ge k}f_j)^*$ in ${\cal O}$
and $(f_j)$ converges to  $f$ q-e.
\end{cor}

\begin{proof}
Let $z\in U$. By Theorem \ref{thm2.7} there exists
a plurifine open ${\cal O}$ containing $z$, a constant $M>0$, a sequence of open sets $(\omega_j)$ with ${\cal O} \subset\omega_j\subset \Omega$, and
functions
$f_j\in \PSH(\omega_j)$ with $f_j\le M$ on $\omega_j$, such that $(f_j)$ converges
q-e to $u$ in $U$. Therefore $u=\inf_k(\sup_{j\ge k}f_j)$ q-e. Because  $\inf_k(\sup_{j\ge k}f_j)=
\inf_k[(\sup_{j\ge k}f_j)^*]$ q-e. in ${\cal O}$ we conclude that $u=\inf_k[(\sup_{j\ge k}f_j)^*]$
q-e. As the functions $f_j$ are bounded from above by $M$ on $\cal O$, the function $\inf_k[(\sup_{j\ge k}f_j)^*]$ belongs to
$\FPSH(\cal O)$ in view of \cite[theorem 3.9]{EFW}. It follows that $u=\inf_k[(\sup_{j\ge k}f_j)^*]$ everywhere on ${\cal O}$.
\end{proof}

\begin{definition}\label{def3.3} Cf.~\cite[Definition 2.2]{EFW}. A function $u:U\longrightarrow [-\infty,+\infty[$
is called \emph{$C$-strongly finely plurisubharmonic on $U$}, $u\in \FCPSH(U)$,  if for every point
$z$ in $U$ one can find a compact ${\cal F}$-neighborhood  $K$ of $z$ and continuous plurisubharmonic functions $f_j$ defined on open neighborhoods of $K$ such that $u=\lim f_j$ uniformly on $K$.
A function $u:U\longrightarrow [-\infty,+\infty[$ is called {\em strongly $\calf$-plurisubharmonic} if $u$ is the pointwise limit of a decreasing net of $\FCPSH$ functions.
\end{definition}

It is clear that every $u\in\FCPSH(U)$
belongs to $\FPSH(U)$.

\begin{theorem}\label{thm2.8} Let $u$ be an ${\cal F}$-plurisubharmonic function
in $U$. Then there exists a  pluripolar ${\cal F}$-closed set $E$ such that
$u$ is $C$-strongly ${\cal F}$-plurisubharmonic in $U\setminus E$.
\end{theorem}

For the proof of Theorem \ref{thm2.8} we need the following lemma.

\begin{lemma} \label{lem2.9} Let $(u_j)$ be a sequence of
 ${\cal F}$-plurisubharmonic functions that is ${\cal F}$-locally uniformly
bounded from above on an ${\cal F}$-open $V\subset \Omega$.
Then there exists a pluripolar set $E\subset V$ such that every
 point $z$ of $V\setminus E$ admits a Euclidean compact ${\cal F}$-neighborhood
$K_z$, such that for every $j$ the restriction
of $u_j$ to $K_z$ is continuous.
\end{lemma}

\begin{proof} Let $a\in V$. Following the proof of \cite[Theorem 2.4]{EFW}, there exist an open ball  $B=B(a,r)$ containing
an ${\cal F}$-open neighborhood ${\cal O_a}$ of $a$, a sequence $(v_j)$ of
uniformly bounded plurisubharmonic functions
on $B$, and a bounded $\varphi\in\psh(B)$, such that for all $j$,
$u_j=v_j-\varphi$ on ${\cal O_a}$.

 Let $C$ be an upper bound for the
$v_j$ on $B$. Put $w_j=(v_j-C)/(C-v_j(a))\le 0$.
We apply the quasi-Brelot property of plurisubharmonic functions, \cite[Theorem 3.3]{EW2},  to the plurisubharmonic function $\varphi-\sum_j 2^{-j}w_j$. It states that there exists
a  pluripolar set $E_a\subset V$
such that every point $t$ of $B\setminus E_a$ admits an ${\cal F}$-neighborhood
$K_t$ that is compact in the Euclidean topology and such that the restriction
of $\varphi-\sum_j 2^{-j}w_j$ to $K_t$ is continuous. Because the functions $w_j$ and $\varphi$ are upper semicontinuous and negative on
$K_t$,
we infer that the restriction of $\varphi$ to $K_t$ is continuous, and
for every $j$ the restriction
of $w_j$, and hence of $u_j$,  to $K_t$ is continuous too. The
 quasi-Lindel\"of property of the  plurifine topology provides us with
a pluripolar set $F$ and a sequence of points
$(a_k)$ in $V$ such that $V=\cup_k {\cal O}_{a_k}\cup F$. Now let
$E$ be  the  ${\cal F}$-closure of the pluripolar set
 $\cup_kE_{a_k}\cup F$. Then $E$ is pluripolar and every point $z$ of $V\setminus E$
admits an ${\cal F}$-neighborhood which is Euclidean compact and meets the conditions of the lemma.
\end{proof}

\begin{proof}[Proof of Theorem \ref{thm2.8}] Let $z\in U$. In view of Theorem \ref{thm2.7} and its
Corollary \ref{cor2.7}, we can find an ${\cal F}$-open ${\cal O}_z$ containing $z$
and a bounded from above sequence $(f_j)$ of plurisubharmonic functions defined on open neighborhoods of
${\cal O}_z$ such that $u  =\inf_k [(\sup_{j\ge k} f_j)^*]$ on ${\cal O}_z$.
Hence there exists an $\calf$-closed pluripolar subset $P$ of ${\CC}^n$ such that
$\inf_k(\sup_{j\ge k}f_j)=\inf_k[(\sup_{j\ge k}f_j)^*]$ on ${\cal O}_z\setminus P$.
By Lemma \ref{lem2.9},  one can also find a
pluripolar ${\cal F}$-closed set $Q$ such that every point
$t\in ({\cal O}_z\setminus P)\setminus Q$ admits a compact ${\cal F}$-neighborhood
$K_t$ on which the restriction of the functions $u$, $(\sup_{j\ge k}f_j)^*$, and
$\max_{k\le j\le i}f_j$, $1\le k\le i$, are Euclidean continuous. Let
$t\in {\cal O}_z\setminus (P\cup Q)$ and let $L_t$ be a compact ${\cal F}$-neighborhood
 of $t$ such that  $L_t\subset K_t\cap ({\cal O}_z\setminus (P\cup Q))$.
The sequences  $(\sup_{j\ge k}f_j)_k$ and
$(\max_{k\le j\le i}f_j)_i$ are monotonic, hence by Dini's Theorem, we have that  $\lim_{k\to +\infty}\sup_{j\ge k}f_j=u$, and
for all $k$,
$\lim_{i\to +\infty}\max_{k\le j\le i}f_j=\sup_{j\ge k}f_j$ uniformly
on $L_t$. From this we infer that the restrictions of functions
$f_{k,i}=\max_{k\le j\le i}f_j$, $k\le i$, which are plurisubharmonic in suitable neighborhoods of
$L_t$, approximate  $u$ uniformly on $L_t$.
\end{proof}

\section{\label{sec1} The Monge-Amp\`ere operator for
${\cal F}$-plurisubharmonic functions}

Let $U\subset\CC^n$ be a  plurifine domain and let $u$ be the restriction to $U$ of a function
$v\in {\PSH}(\Omega)\cap L_{loc}^\infty$, where $\Omega$ is an open neighborhood of $U$ in $\CC^n$. Then $u$ is ${\cal F}$-plurisubharmonic in
 $U$ and it is natural to try and define on $U$ the
 Monge-Amp\`ere of $u$ as the
 measure $(dd^cv)^n$ restricted to $U$.
For this definition to make  sense, the
 measure $(dd^cu)^n$ thus defined, should not depend on the
choice of $v$. This is guaranteed by the following result of Bedford and Taylor, \cite{BT}:

\begin{theorem}[{\cite[Corollary 4.3]{BT}}]\label{thm1.1} Let $u$ and $v$
be locally bounded plurisubharmonic functions on a domain $\Omega\subset\CC^n$ such that
$u=v$ on a plurifinely open set $V\subset\Omega$. Then $(dd^cu)^n|_V=(dd^cv)^n|_V$.
\end{theorem}

Because of the quasi-Lindel\"of property of the $\calf$-topology and the fact that the Monge-Amp\`ere mass of a bounded plurisubharmonic function does not charge pluripolar sets, Theorem \ref{thm1.1} will enable us to define more generally $(dd^cu)^n$ for functions $u$ that are $\calf$-locally the restriction of functions $v$ as above. See Section \ref{sec3}.

In all that follows, $\Omega$ will be a domain in $\CC^n$ ($n\ge 2$) and $U$ wil be an $\calf$-domain in $\Omega$.

We denote by $\QB(\CC^n)$ the measurable space on $\CC^n$ generated by the Borel sets and the
 pluripolar subsets of ${\CC}^n$ and by $\QB(U)$ the trace of
$\QB({\CC}^n)$ on $U$.

Let $u$ be a finite ${\cal F}$-plurisubharmonic function on a $\calf$-domain $U$. Then by \cite[Theorem 3.1]{EW2}, $u$ is ${\cal F}$-locally the difference
$w=v_1-v_2$ of bounded plurisubharmonic functions $v_1,v_2$ defined on an open set in $\CC^n$.
As a consequence ${\cal F}$-plurisubharmonic functions are $\calf$-continuous and therefore ${\cal F}$-plurisubharmonic functions are $\calf$-locally bounded on $U$ if and only if they are finite on $U$.

Following Cegrell and Wiklund \cite{CW} one  defines the  (signed) Monge-Amp\`ere mass $(dd^cw)^n$
associated to  $w$ as
\begin{equation}\label{CeWi} (dd^cw)^n=\sum_{p=0}^nC_n^p(-1)^p(dd^cv_1)^p\wedge(dd^cv_2)^{n-p},
\end{equation} where $C_n^p=\binom{n}{p}$. Therefore we would like to define $(dd^c u)^n$ ${\cal F}$-locally by
\[(dd^cu)^n=(dd^c w)^n.\]
To do so we have to generalize the Bedford and Taylor result,  Theorem \ref{thm1.1}.
We need some auxiliary results.

\begin{lemma}\label{lem1.2} Let $(u_1^k),\ldots,(u_n^k)$ and  $(v_1^k),\ldots,(v_n^k)$ be
monotonically decreasing (or increasing) sequences of locally bounded plurisubharmonic functions on an open set $\Omega\subset\CC^n$ that converge to functions $u_1,\ldots,u_n$ and
$v_1,\ldots,v_n$ that are locally bounded on $\Omega$. Let
${\cal O}\subset \Omega$ be an ${\cal F}$-open set. Suppose that for all $k$
\[dd^cu_1^k\wedge...\wedge dd^cu_n^k|_{\cal O}=
dd^cv_1^k\wedge...\wedge dd^cv_n^k|_{\cal O}.
\]
Then
\[
dd^cu_1\wedge...\wedge dd^c u_n|_{\cal O}=
dd^cv_1\wedge...\wedge dd^c v_n|_{\cal O}.
\]
\end{lemma}

\begin{proof} The proof consists of a straightforward adaptation of \cite[Lemma 4.1]{BT} and will be omitted.
\end{proof}

\begin{prop}\label{prop1.3} Let $u_1,...,u_n, v_1,...,v_n$ be locally bounded plurisubharmonic
functions on  $\Omega$, let $0\le m \le n$, and let
${\cal O}\subset\{u_1>v_1\}\cap\cdots\cap\{u_m>v_m\}$ be $\calf$-open. Then
\[\begin{split}
dd^c\max(u_1,v_1)\wedge...\wedge dd^c\max(u_m,v_m)\wedge dd^cu_{m+1}...\wedge dd^cu_n|_{\cal O}\\=
dd^cu_1\wedge...\wedge dd^cu_n|_{\cal O}.
\end{split}\]
\end{prop}
\begin{proof} We proceed as in \cite{BT}. In case ${\cal O}$ is a Euclidean open set the statement is obvious. Let $(u_1^k),\ldots,(u_m^k)$ be decreasing sequences of smooth plurisubharmonic functions that converge, respectively,
to $u_1,\ldots,u_m$.
We put ${\cal O}_k=\cap_{i=1}^m\{u_i^k>v_i\}$. Then for every $k$ the set ${\cal O}_k$
is a Euclidean open set on which $\max(u_i^k,v)=u_i^k$, and ${\cal O}\subset {\cal O}_k$.
Therefore
\[\begin{split}
dd^c\max(u_1^k,v_1)\wedge...\wedge dd^c\max(u^k_m,v_m)\wedge dd^cu_{m+1}...\wedge dd^cu_n|_{{\cal O}}\\ =
dd^cu_1^k\wedge...\wedge dd^cu^k_m\wedge dd^cu_{m+1}...\wedge dd^cu_n|_{\cal O}.
\end{split}
\]
From this and Lemma \ref{lem1.2} we conclude
\[\begin{split}dd^c\max(u_1,v_1)\wedge...\wedge dd^c\max(u_m,v_m)\wedge dd^cu_{m+1}\wedge... dd^c\max(u_n,v_n)|_{\cal O}\\=
dd^cu_1\wedge...\wedge dd^cu_n|_{\cal O}.
\end{split}\]
\end{proof}

\begin{cor}\label{cor1.4} Let $u_1,\ldots,u_n, v_1,\ldots,v_n$ be locally bounded plurisubharmonic functions on $\Omega$ and let ${\cal O}\subset \Omega$ be $\calf$-open.
If $u_1=v_1$, ..., $u_n=v_n$ on ${\cal O}$, then
\[
dd^cu_1\wedge...\wedge dd^cu_n|_{\cal O}=dd^cv_1\wedge...\wedge dd^cv_n|_{\cal O}.
\]
\end{cor}

\begin{proof} Let $k >0$. As
$u_i=\max(u_i,v_i-\frac{1}{k})$, $i=1,...,n$ on $\cal O$, we have by Proposition \ref{prop1.3}
\[
dd^cu_1\wedge...\wedge dd^cu_n|_{\cal O}=
dd^c\max(u_1,v_1-\frac{1}{k})\wedge...\wedge dd^c\max(u_n,v_n-\frac{1}{k})|_{\cal O}.
\]
Letting $k\to \infty$, we obtain from Lemma \ref{lem1.2}
\[dd^cu_1\wedge...\wedge dd^cu_n|_{\cal O}=
dd^c\max(u_1,v_1)\wedge...\wedge dd^c\max(u_n,v_n)|_{\cal O}.
\]
Similarly we find
\[
dd^cv_1\wedge...\wedge dd^cv_n|_{\cal O}=
dd^c\max(u_1,v_1)\wedge...\wedge dd^c\max(u_n,v_n)|_{\cal O},
\] which completes the proof.
\end{proof}

\begin{prop}\label{prop1.5} Let $f_1,g_1, f_2,g_2\in \PSH(\Omega)\cap L_{loc}^\infty$
and let ${\cal O}=\{f_1-g_1>f_2-g_2\}$. Then
\[(dd^c\max(f_1-g_1,f_2-g_2))^n|_{\cal O}=(dd^c(f_1-g_1))^n|_{\cal O}.
\]
\end{prop}

\begin{proof}
We will use equation \eqref{CeWi} and apply Proposition \ref{prop1.3} with various $m$ and for every $j$, $u_j=f_1+g_2$ and $v_j=f_2+g_1$. Then

\begin{align*}
(dd^c\max&(f_1-g_1,f_2-g_2))^n|_{\cal O}\\&=(dd^c\max(f_1+g_2,f_2+g_1)-(g_1+g_2))^n|_{\cal O}
\\
&= \sum_{p=0}^n(-1)^pC_n^p(dd^c\max(f_1+g_2,f_2+g_1))^{n-p}\wedge (dd^c(g_1+g_2))^p|_{\cal O}\\
&
= \sum_{p=0}^n(-1)^pC_n^p(dd^c(f_1+g_2))^{n-p}\wedge (dd^c(g_1+g_2))^p|_{\cal O}\\
&=(dd^c(f_1-g_1))^n|_{\cal O}.
\end{align*}
\end{proof}
Now we can prove a generalization of Theorem \ref{thm1.1}.

\begin{theorem}\label{thm1.6} Suppose that $u_1,u_2,v_1,v_2$ are plurisubharmonic
functions on a domain
$\Omega\subset\CC^n$. If $u_1-u_2=v_1-v_2$ on an $\calf$-open ${\cal O}\subset \Omega$, then
$(dd^c(u_1-u_2))^n|_{\cal O}=(dd^c(v_1-v_2))^n|_{\cal O}$.
\end{theorem}
\begin{proof} Let $k$ be a positive integer. As $u_1-u_2>v_1-v_2-\frac{1}{k}$
in ${\cal O}$, we have in view of Proposition \ref{prop1.5}
\begin{align*}(dd^c&(u_1-u_2))^n|_{\cal O}=
(dd^c\max(u_1-u_2,v_1-v_2-\frac{1}{k}))^n|_{\cal O}\\
&=(dd^c[\max(u_1+v_2,u_2+v_1-\frac{1}{k})-(u_2+v_2)])^n|_{\cal O}\\
&=\sum_{p=0}^nC_n^p(-1)^p(dd^c\max(u_1+v_2,u_2+v_1-\frac{1}{k}))^{n-p}\wedge(dd^c(u_2+v_2))^p|_{\cal O}
\end{align*}
We let $k\to \infty$  and obtain in view of Lemma \ref{lem1.2},
\begin{align*}(dd^c&(u_1-u_2))^n|_{\cal O}\\&=
\sum_{p=0}^nC_n^p(-1)^p(dd^c\max(u_1+v_2,u_2+v_1))^{n-p}\wedge(dd^c(u_2+v_2))^p|_{\cal O}\\
&=(dd^c\max(u_1-u_2,v_1-v_2))^n|_{\cal O}.
\end{align*}
Similarly we have
\[(dd^c(v_1-v_2))^n|_{\cal O}=
(dd^c\max(u_1-u_2,v_1-v_2))^n|_{\cal O},
\]
and the proof is complete.
\end{proof}

Theorem \ref{thm1.6} will enable us in Section \ref{sec3} to define the Monge-Amp\`ere operator for bounded
 ${\cal F}$-plurisubharmonic functions.

\section{\label{sec3} Positivity of the Monge-Amp\`ere operator for $\calf$ plurisubharmonic functions}

In this section we will show that the Monge-Amp\`ere mass of a finite
${\cal F}$-plurisubharmonic function defined on an ${\cal F}$-open set can be defined and is positive.

\begin{lemma}\label{lem3.1}  Let $(f_j)$ be a monotonically increasing, respectively decreasing,  sequence of $\calf$-plurisubharmonic functions that is $\calf$-locally uniformly bounded from above, respectively from below, on an $\calf$-open set $U\subset \Omega$.
Then for all $z\in U$, one can find an
${\cal F}$-neighborhood ${\cal O}$ of $z$ contained in an open set $B\subset\Omega$, a plurisubharmonic function $\Phi$ on $B$, and an increasing, respectively decreasing, sequence of plurisubharmonic functions
$(u_j)$ on $B$ such that
$f_j=u_j-\Phi$ on ${\cal O}$.
\end{lemma}

\begin{proof} The lemma follows immediately from the proof of \cite[Theorem 2.4]{EFW}.
\end{proof}

\begin{remark} The lemma remains valid for increasing, respectively decreasing, directed families of $\calf$-plurisubharmonic functions that are $\calf$-locally bounded from above, respectively below.
\end{remark}
The proof of the following lemma is inspired by \cite[Lemma 4.1]{BT}.
\begin{lemma} \label{lem3.2} Let $\Omega$ be open in
${\CC}^n$, $n\ge 1$, let ${\cal O}\subset \Omega$ be a plurifine open subset, and let
$(u_j^1)$ and $(u_j^2)$ be two monotone sequences of plurisubharmonic
functions that are  bounded in $\Omega$, each converging to a bounded plurisubharmonic
function $u^1$, respectively $u^2$ on $\Omega$.
If $(dd^c(u_j^1-u_j^2))^n|_{\cal O}\ge 0$, then
$(dd^c(u_1-u_2))^n|_{\cal O}\ge 0$.
\end{lemma}

\begin{proof} Observe that $(dd^c(u_j^1-u_j^2))^n$ and $(dd^c(u_1-u_2))^n$ are defined on $\Omega$ by \eqref{CeWi}. By the quasi-Lindel\"of property of the plurifine topology we can write  ${\cal O}$ as a pluripolar set $E$ joined with a countable union of $\calf$-open sets of the form $B\cap \{\psi>0\}$, where $B$ is an open ball in
 ${\CC}^n$, $B\subset {\overline B}\subset \Omega$, and
$\psi$ plurisubharmonic in a neighborhood of $\overline B$. As
pluripolar sets are neglegible for the Monge-Amp\`ere mass of bounded plurisubharmonic
 functions, it suffices to show the result for
 ${\cal O}=B\cap \{\psi>0\}$. We proceed as follows. Multiplying $\psi$ with a smooth cut-off function that is positive on $B$, we find
an  ${\cal F}$-continuous function ${\tilde \psi}\ge 0$  with compact support such that
${\cal O}=\{z\in \Omega; \tilde \psi(z)>0\}$.
Then, according to \cite[Theorem 3.2]{BT}, cf.\ also \cite[Lemma 4.1]{BT},
\[0\le \lim_{j\to \infty}\int {\tilde \psi}(dd^c(u_j^1-u_j^2))^n=
\int  {\tilde \psi}(dd^c(u_1-u_2))^n.\]
 When we replace ${\tilde \psi}$ by
$f{\tilde \psi}$, where $f$ is any
${\cal F}$-continuous functions $\ge 0$ on $\Omega$, this inequality remains valid. That is,
$\int f{\tilde \psi}(dd^c(u_1-u_2))^n\ge 0$. Hence
${\tilde \psi}(dd^c(u_1-u_2))^n\ge 0$, and it follows that
$(dd^c(u_1-u_2))^n|_{\cal O}\ge 0$.
\end{proof}


\begin{theorem}\label{thm3.4} Let $f$ be in
$\FCPSH$ on an ${\cal F}$-open set $U$. Then $(dd^cf)^n$ can be defined on $U$ as a positive measure.
\end{theorem}

\begin{proof} Let $z\in U$. There exists a compact ${\cal F}$-neighborhood $K=K_z$ of $z$
and functions $f_1,f_2,\ldots$ that are
finite, continuous and plurisubharmonic in open neighborhoods
$O_1,O_2,\ldots$ of $K$, such that the sequence
$(f_j|_K)$ converges to $f$ uniformly on $K$. Take any such $K$ and $f_j$.
The functions $f_j$ are continuous, hence after shrinking the open $O_j$ and adapting the functions $f_j$ by small constants, we can assume that for every $j$
$O_{j+1}\subset O_j$ and $f_{j+1}\le f_j$ on $O_{j+1}$.
Next it follows from Lemma \ref{lem3.1} that we can find an open
$B$ containing $z$,  two functions
$u$ and $\Phi$ in $\psh(B)$ and a decreasing sequence  $(u_j)$ of
 plurisubharmonic functions on $B$ such that on some $\calf$-neighborhood ${\cal O}_z\subset B\cap U$
$f_j=u_j-\Phi$ and $f=u-\Phi$. For every $j$, we have
$(dd^c(u_j-\Phi))^n|_{{\cal O}_z}\ge 0$. As the sequence $(u_j)$
decreases we conclude in view of Lemma \ref{lem3.2} that
\[\label{eqkr1} (dd^cf)^n|_{{\cal O}_z}=(dd^c(u-\Phi))^n|_{{\cal O}_z}\ge 0.\]

Theorem \ref{thm1.6} shows that $(dd^cf)^n|_{{\cal O}_z}$ is independent of the choice of $U$ and $\Phi$.
Now, by the
quasi-Lindel\"of property of the plurifine topology, there exist a sequence
$(z_j)$ of points in $U$ and a pluripolar subset
$P$ of $U$ such that
\[U=\cup_j{\cal O}_{z_j}\cup P.\] Let $u^{z_j}$, $\Phi^{z_j}$
be the corresponding plurisubharmonic functions.
Because the measures $(dd^c(u^{z_j}-\Phi^{z_j}))^n$ do not
have mass on pluripolar sets, we can define a signed measure  $\mu$ on
$U$ by putting for every $A\in \QB(U)$,
\begin{equation}\label{eq1.6} \mu(A)=\int_{{\cal O}_{z_0}}1_A(dd^cu)^n+
\sum_{k\ge 1}\int_{{\cal O}_{z_k}\setminus (\cup_{0\le i\le k-1} {\cal O}_{z_i})}
1_A(dd^cu)^n.
\end{equation}
By Theorem  \ref{thm1.6}
the measure $\mu$ defined by \eqref{eq1.6}
$A\in \QB(U)$ is independent of the choice of the sequence
$({\cal O}_{z_j})$ and the functions $u^{z_j}$ and $\Phi^{z_j}$.
\end{proof}
\begin{definition}\label{def1.6}
We denote the signed measure $\mu$ on the $\QB(U)$ thus defined, by
$(dd^cu)^n$ and call it the Monge-Amp\`ere measure associated to $u$.
The operator $u\in \FCPSH(U)\mapsto
(dd^c u)^n$ will be called the
Monge-Amp\`ere operator.
\end{definition}

\begin{theorem}\label{thm3.5} Let $U$ be an $\cal F$-open set and let $f\in \FPSH(U)$ be finite. Then $(dd^cf)^n$ can be defined and is a positive measure on $U$.
\end{theorem}

\begin{proof} According to Theorem \ref{thm2.8}, there exists a
pluripolar ${\cal F}$-closed $F\subset \Omega$ such that
$f$ is $C$-strongly ${\cal F}$-plurisubharmonic on  $U\setminus F$. Then by the previous theorem,  $(dd^cf)^n$ to
$U\setminus F$ is a positive measure. Because pluripolar sets will not be charged by $(dd^c u-\Phi)^n$ when $u$, $\phi$ are bounded plurisubharmonic functions,  we can define $(dd^cf)^n(F)=0$. The resulting measure is independent of the choice of $F$. We conclude that $(dd^cf)^n$  is a well defined positive measure.
\end{proof}

\begin{remark}\label{app4.10} Let $f$ be a finite plurisubharmonic function
on a  Euclidean domain $\Omega$ in ${\CC^n}$. Then $f$ is in particular in $\FPSH(\Omega)$ and finite. It is a consequence of the previous theorem that the
Monge-Amp\`ere mass $(dd^cf)^n$ is a well defined positive measure, which is easily seen to coincide with the
\emph{nonpolar part} $NP(dd^cf)^n$, of the Monge-Amp\`ere measure defined in \cite[P.236]{BT}. This measure is in general not a Radon measure, i.e.\ not Euclidean locally finite. It is, however, $\calf$-locally finite. In general $f$ may well not belong to ${\cal D}(\Omega)$, the domain of MA in the sense of B\l ocki
cf.~\cite{B}. In fact, \AA hag, Cegrell and Hiep have an example of a finite subharmonic function which is not in ${\cal D}(\Omega)$, cf.~\cite{Cegrell}.

Completely analogous to Bedford and Taylor, cf.~\cite[P.236]{BT} we may define the \emph{non-polar part} $NP(dd^cf)^n$ as zero on $\{f=-\infty\}$ and by
\[\int_E NP(dd^cf)^n=\lim_{j\to\infty}\int_E (dd^c(\max(f,-j))^n.\]
\end{remark}

The following  theorem extends the result of Bedford and Taylor \cite[Theorem 1.1]{BT} to
${\cal F}$-plurisubharmonic functions on ${\cal F}$-open sets.

\begin{theorem}\label{thm1.7} Let $u$ and $v$ be two finite
${\cal F}$-plurisubharmonic functions on an ${\cal F}$-open set $U\subset\Omega$, with $\Omega$ open in $\CC^n$. If $u=v$ on a ${\cal F}$-open ${\cal O}\subset U$, then
\[(dd^cu)^n|{\cal O}=(dd^cv)^n|{\cal O}.\]
\end{theorem}

\begin{proof} Because of \cite[Theorem 3.1]{EW2} and the
quasi-Lindel\"of property of the plurifine topology
we can find bounded plurisubharmonic functions
$u_j^1$, $u_2^j$, $v_1^j$, and
$v_2^j$ defined on open neighborhoods of compact sets $K_j\subset\Omega$, ($j=1,2, \ldots$)  and a
pluripolar set $P$ such that ${\cal O}=\cup_j K_j\cup P$, and for every
$j$, $u=u_1^j-u_2^j$ and $v=v_1^j-v_2^j$  on an ${\cal F}$-open neighborhood ${V}_j$ of $K_j$.
Now we have
$(dd^c(u_1^j-u_2^j))^n|K_j=(dd^c(v_1^j-v_2^j))^n|K_j$
according to Theorem \ref{thm1.6}, and the proof is complete.
\end{proof}

\end{document}